\documentclass[11pt]{article} 
\usepackage[margin=1.4in]{geometry}
\usepackage[english]{babel}
\usepackage[T1]{fontenc}
\usepackage{lmodern} 
\usepackage{graphicx}
\usepackage{wrapfig}
\usepackage{subfig}
\usepackage{caption}
\usepackage{paralist}

\usepackage{etoolbox}
\newtoggle{long}
\toggletrue{long}

\newcommand{\titel}{A Tight Bound for Minimal Connectivity}

\usepackage{algorithm} 
\usepackage[noend]{algpseudocode} 

\usepackage[usenames]{color} 
\definecolor{hellblau}{rgb}{0.2,0.4,1} 
\definecolor{dunkelblau}{rgb}{0,0,0.8}
\definecolor{dunkelgruen}{rgb}{0,0.5,0}
\usepackage[
	pdftex,
	colorlinks,
	linkcolor=dunkelblau,
	urlcolor=dunkelblau,
	citecolor=dunkelgruen,
	bookmarks=true,
	linktocpage=true,
	pdftitle={\titel},
	pdfauthor={Jens M. Schmidt},
	pdfsubject={}
]{hyperref} 
\usepackage{xspace}
\usepackage{pdfpages}

\usepackage{amsmath} 
\usepackage{amsthm} 
\usepackage{amsfonts} 
\theoremstyle{plain} 
	\newtheorem{satz}{Satz}[] 
	\newtheorem{theorem}[satz]{Theorem}
	\newtheorem{lemma}[satz]{Lemma}
	
	\newtheorem{proposition}[satz]{Proposition}
\theoremstyle{remark} 
	\newtheorem*{remark}{\textbf{Remark}} 
\theoremstyle{definition} 
	\newtheorem{definition}[satz]{Definition}
	\newtheorem{corollary}[satz]{Corollary}
	\newtheorem*{conjecture}{Conjecture}
	\newtheorem{example}[satz]{Example}

\begin{document}
	\title{\titel}
		\author{Jens M. Schmidt\footnote{This research was supported by the DFG grant SCHM 3186/1-1.}\\Institute of Mathematics\\TU Ilmenau, Germany}
	\date{}
	\maketitle

\begin{abstract}
For minimally $k$-connected graphs on $n$ vertices, Mader proved a tight lower bound for the number $|V_k|$ of vertices of degree $k$ in dependence on $n$ and $k$. Oxley observed 1981 that in many cases a considerably better bound can be given if $m := |E|$ is used as additional parameter, i.e.\ in dependence on $m$, $n$ and $k$. It was left open to determine whether Oxley's bound is best possible.

We show that this is not the case, but propose a closely related bound that deviates from Oxley's long-standing one only for small values of $m$. We prove that this new bound is best possible. The bound contains Mader's bound as special case.
\end{abstract}

\section{Introduction}
\emph{Minimally $k$-connected} graphs (i.e.\ $k$-connected graphs, for which the deletion of any edge decreases the connectivity) have been in the focus of both structural and extremal graph theory~\cite{Bollobas2004,Kriesell2013} since their early days. 
For $k=2$, Dirac~\cite{Dirac1967} and Plummer~\cite{Plummer1968} showed that every minimally $2$-connected graph contains a vertex of degree 2. In 1969, Halin~\cite{Halin1969b} generalized this result by proving that every minimally $k$-connected graph contains a vertex of degree $k$. This proof led to a plethora of further results about the structure of minimally $k$-connected graphs in general, and the number $|V_k|$ of vertices of degree $k$ of these graphs in particular (see~\cite{Mader1996} for an extensive survey), which 1979 eventually culminated in a tight lower bound for $|V_k|$ shown by Mader~\cite{Mader1979}.

While Mader proved that his bound is tight for all $n := |V|$ and $k$ (up to certain parity values), Oxley~\cite{Oxley1981} found shortly after, and inspired by matroids, a different lower bound for $|V_k|$ that uses the parameters $m$, $n$ and $k$. Oxley states 1981 that his bound ``\emph{frequently sharpens}'' Mader's~\cite{Oxley1981}. Yet, the problems of classifying the parameters for which Oxley's bound improves Mader's and, more importantly, of finding a lower bound that is generally best possible in dependence on $m$, $n$ and $k$, are still open since then.

We solve both questions by proposing a bound that consists of Oxley's bound if $m \geq k(kn-1)/(2k-1)$ and a new simple bound if $m < k(kn-1)/(2k-1)$. The new bound contains Mader's bound as special case when $m=k(kn-1)/(2k-1)$, and strictly improves the best-known specialized lower bounds for $k \in \{2,3\}$ given in~\cite{Gross2013,Oxley1981}. As main result, we show that our bound is best possible for all $m$, $k \geq 2$ and $n > 2k$ (up to certain parity values).

The difficult part of the result is to exhibit suitable infinite families of minimally $k$-connected graphs in order to prove tightness for both ranges of $m$ mentioned above. The infinite family that we construct for small $m$ may also be of interest in different problem settings, as it consists of minimally $k$-connected graphs that are ``almost $k$-regular'', i.e.\ such that $m$ is close to a prescribed value slightly above the (minimal possible) value $\lceil kn/2 \rceil$.

After giving the necessary preliminaries, we will revisit and generalize existing lower bounds on $|V_k|$ in Section~\ref{sec:revisiting} and then formulate the new bound and prove its tightness in Section~\ref{sec:tight}.

\section{Preliminaries}\label{sec:prel}
We consider only finite, simple and undirected graphs. For a graph $G=(V,E)$, let $n_G := |V|$ and $m_G := |E|$; if $G$ is clear from the context, we will omit the subscript. A $k$-\emph{separator} of a graph is a set of $k \geq 0$ vertices whose deletion leaves a disconnected graph. A graph $G$ is $k$-\emph{connected} if $n > k$ and $G$ contains no $(k-1)$-separator.
A $k$-connected graph $G$ is \emph{minimally $k$-connected} if $G-e$ is not $k$-connected for every edge $e \in E$. Since every non-empty graph is $0$-connected according to this definition, there is no minimally $0$-connected graph that contains at least one edge. We thus assume $k \geq 1$ throughout this paper.

For a graph $G$, let $V_k := V_k(G)$ be the set of vertices of degree $k$ and let $E_k$ be the set of edges in $G$ that is induced by $V_k$. Further, let $F := F(G) := G-V_k$ and let $c_F$ be the number of components of $F$. If $G$ is minimally $k$-connected, the following lemmas by Mader ensure that $F$ carries a very special structure.

\begin{lemma}[{\cite[Korollar~1]{Mader1972}}]\label{lem:ForestLemma}
For every minimally $k$-connected graph, $F$ is a forest.
\end{lemma}

\begin{lemma}[{\cite[p.~66]{Mader1979}}]\label{lem:cplusEk}
For every minimally $k$-connected graph, $c_F+|E_k| \geq k$.
\end{lemma}

We abbreviate $a \equiv b \text{ (mod } c)$ as $a \equiv_c b$ and write the statement that $a \equiv_c b$ for some $b \in \{b_1,\ldots,b_t\}$ as $a \equiv_c b_1,\ldots,b_t$.

\section{Revisiting the old Bounds}\label{sec:revisiting}
Let $G$ be a minimally $k$-connected graph. We revisit, generalize and compare the lower bounds for $|V_k|$ that are already known.

Mader showed that $|V_k| \geq k+1$ and $|V_k| \geq \Delta$~\cite[Korollar~2 and~Satz~4]{Mader1972}. Clearly, the latter bound is at least as good as the former, unless $G$ is $k$-regular (in which case $|V_k| = n$). However, both bounds are far from being tight.

In his seminal paper~\cite[Satz~3]{Mader1979}, Mader eventually proved
\begin{align}
	|V_k| \geq \frac{(k-1)n+2k}{2k-1} \label{eq:Mader}
\end{align}
and showed that there is a minimally $k$-connected graph attaining equality in~\eqref{eq:Mader} for every $k$ and $n > 2k$ such that $n \equiv_{2k-1} 0,1,2,3,5,7,\ldots,2k-3$. In that sense, Bound~\eqref{eq:Mader} is tight for the parameters $n$ and $k$. The following is a slight generalization of Bound~\eqref{eq:Mader}, which relates it to $\Delta$.

\begin{theorem}\label{thm:mader}
For every minimally $k$-connected graph,
\begin{align}
	|V_k| \geq \frac{(k-1)n+2(c_F+|E_k|)+\max\{0,\Delta-(k+1)\}}{2k-1}. \label{eq:MaderGeneral}
\end{align}
\end{theorem}
\begin{proof}
There are exactly $|E(F)| = |V(F)|-c_F = n-|V_k|-c_F$ edges in $F$. Thus, the number of edges that have exactly one end vertex in $F$ is at least $(k+1)|V(F)|-2|E(F)|+\max\{0,\Delta-(k+1)\} = (k-1)(n-|V_k|)+2c_F+\max\{0,\Delta-(k+1)\}$. Counting these edges in dependence on $V_k$, we obtain $k|V_k|-2|E_k| \geq (k-1)(n-|V_k|)+2c_F+\max\{0,\Delta-(k+1)\}$, which gives the claim.
\end{proof}

According to Lemma~\ref{lem:cplusEk}, Bound~\eqref{eq:MaderGeneral} implies Bound~\eqref{eq:Mader}. Although Bound~\eqref{eq:Mader} is tight for many graphs, it is far from being tight if $m$ is introduced as additional parameter. In fact, we will show in the next section that Bound~\eqref{eq:Mader} is only best possible when $m = \frac{k(kn-1)}{2k-1}$.

Using a surprisingly simple proof, Oxley~\cite[Prop.~2.19]{Oxley1981}\cite[Fact~74 in~6.6.12]{Gross2013} observed for $k \geq 2$ that $|V_k| \geq \frac{m-n+1}{k-1}$. For the parameters $m$, $n$ and $k$, this is the best bound known so far. Since Oxley used $c_F+|E_k| \geq 1$ in his proof, the bound can be slightly strengthened by applying Lemma~\ref{lem:cplusEk}. In addition, a closer look at the proof of the bound shows that we can actually obtain the following \emph{equality} for $V_k$.

\begin{theorem}\label{thm:Oxley}
For $k \geq 2$ and every minimally $k$-connected graph,
	\begin{align}
		|V_k| &= \frac{m-n+c_F+|E_k|}{k-1}. \label{eq:OxleyGeneral}\\
		\text{In particular, }|V_k| &\geq \left\lceil \frac{m-n+k}{k-1} \right\rceil.\phantom{In particular, } \label{eq:Oxley}
	\end{align}
\end{theorem}
\begin{proof}
The number of edges that are not in $F$ is $k|V_k|-|E_k|$, as $k|V_k|$ double-counts every edge in $E_k$. Hence, $m = k|V_k|-|E_k|+|E(F)|$ and $|V_k| = \frac{m+|E_k|-|E(F)|}{k}$. There are exactly $n-|V_k| = |E(F)|+c_F$ vertices of degree greater than $k$ in $G$, which implies Bound~\eqref{eq:OxleyGeneral}. Bound~\eqref{eq:Oxley} follows from Bound~\eqref{eq:OxleyGeneral} by applying Lemma~\ref{lem:cplusEk}.
\end{proof}

With Bound~\eqref{eq:OxleyGeneral}, we have a bound at hand that is always optimal, as long as a minimally $k$-connected graph with the given parameters exists. Unfortunately, it is not clear at all how to decide whether there is a graph with such a given parameter constellation. We therefore investigate bounds for the rather natural parameters $m$, $n$ and $k$.

The given bounds (apart from~\eqref{eq:OxleyGeneral}, which is always optimal) relate to each other as follows: For the interesting case $n > 2k$, both bounds~\eqref{eq:Mader} and~\eqref{eq:Oxley} imply $|V_k| \geq \lceil k + \frac{1}{2k-1}\rceil = k+1$. The bound $|V_k| \geq \Delta$ however is independent of Bounds~\eqref{eq:Mader},~\eqref{eq:MaderGeneral} and~\eqref{eq:Oxley}: Clearly, $\Delta$ can be smaller than any of these bounds, as e.g.\ the $k$-regular $k$-connected graphs show. For every sufficiently large wheel graph, $\Delta > \eqref{eq:Mader}$ and $\Delta > \eqref{eq:Oxley}$. For every $n > 2k$ and $k > 1$, the graph $K_{k,n-k}$ shows that $ \Delta = n-k > \frac{kn-1}{2k-1} = \eqref{eq:MaderGeneral}$. The next section will show that Bound~\eqref{eq:Oxley} is at least as good as~\eqref{eq:Mader} if and only if $m \geq k(kn-1)/(2k-1)$ (up to parity issues).

\section{A Tight Bound}\label{sec:tight}
Harary~\cite{Harary1962} showed $m \geq \lceil kn/2 \rceil$ for every (minimally) $k$-connected graph, where $m = (kn+1)/2$ can in fact be attained by such graphs when $kn$ is odd.
Mader~\cite[Satz~2]{Mader1971} showed $m \leq kn-\binom{k+1}{2}$ for every minimally $k$-connected graph, where equality is attained only for $K_{k+1}$. Thus, every minimally $k$-connected graph satisfies $\lceil kn/2 \rceil \leq m \leq kn-\binom{k+1}{2}$.

If $m$ is large, our general lower bound for the parameters $m$, $n$ and $k$ will consist of Oxley's Bound~\eqref{eq:Oxley}. If $m$ is small, we use the following lower bound instead, as it outperforms the others in that case. Although the bound is very simple, it does not seem to have been exploited for $|V_k|$ so far.

\begin{lemma}
For every minimally $k$-connected graph,
	\begin{align}
		|V_k| \geq (k+1)n-2m. \label{eq:simple}
	\end{align}
\end{lemma}
\begin{proof}
Let $i \in \mathbb{Q}$ be such that $m = kn/2 + i$. Since $m \geq \lceil kn/2 \rceil$, $i \geq 0$. There are at most $2i$ vertices of degree greater than $k$, as counting the degrees for $n-|V_k| > 2i$ implies the contradiction $m \geq \frac{k|V_k| + (k+1)(n-|V_k|)}{2} > kn/2 + i$. Thus, $n-|V_k| \leq 2i = 2m-kn$, which gives the claim.
\end{proof}

For $k \geq 2$, this gives the general lower bound $|V_k| \geq \max\{(k+1)n-2m, \lceil (m-n+k)/(k-1) \rceil \}$. We next show that this bound is tight.

\begin{theorem}\label{thm:main}
For $k \geq 2$ and every minimally $k$-connected graph $G$,
	\begin{align}
		|V_k| \geq 
		\begin{cases}
			(k+1)n-2m 		& \emph{if } m \leq \frac{k(kn-1)}{2k-1}\\
			\left\lceil (m-n+k)/(k-1) \right\rceil & \emph{if } m \geq \frac{k(kn-1)}{2k-1}.
		\end{cases}
		\label{eq:general}
	\end{align}
The bound is best possible (even without the ceiling) for every $m$, $n \geq 3k-2$ and $k \geq 2$ such that
\begin{compactitem}
	\item $m \equiv_{k(k-1)} k(n-1)-i$ and $0 \leq i \leq 2\lfloor \frac{k}{2} \rfloor$ if $m \leq \frac{k(kn-1)}{2k-1}$, and
	\item $m \equiv_{k-1} k(n-1)$ if $m \geq \frac{k(kn-1)}{2k-1}$.
\end{compactitem}
\end{theorem}
\begin{proof}
Bound~\eqref{eq:general} follows directly from bounds~\eqref{eq:Oxley} and~\eqref{eq:simple}. We prove its tightness under the given assumptions.

Take $k$ vertex-disjoint copies $T_1,\ldots,T_k$ of a tree $T$ with maximum degree at most $k+1$ and $l := |V(T)| \geq 1$. For a vertex $v \in T$, let $v_1,\ldots,v_k$ be the vertices in $T_1,\ldots,T_k$ that correspond to $v$; we call this vertex set the \emph{row} of $v$. Obtain the graph $H_T(k,l)$ from $T_1 \cup \cdots \cup T_k$ by adding $k+1-deg_T(v)$ new vertices for each vertex $v$ in $T$ and joining these vertices to each vertex of the row of $v$ by an edge (see Figure~\ref{fig:ConstructionMader}). This way, every vertex in a tree copy has degree exactly $k+1$ in $H_T(k,l)$, and $F = T_1 \cup \cdots \cup T_k$. The subgraph that is induced by the vertices of a row and the vertices added to this row is called a \emph{layer}. Every layer of $H_T(k,l)$ is a complete bipartite graph.

We will use $H := H_T(k,l)$ in the construction of tight graph families; to simplify later arguments, we first determine $|V_k(H)|$, $n_H$ and $m_H$. Since $|V(F(H))| = kl$ and $|E(F(H))| = k(l-1)$ in $H$, we have $k|V_k(H)| = (k+1)|V(F(H))|-2|E(F(H))| = (k-1)kl+2k$, which implies $|V_k(H)| = (k-1)l+2$ and $n_H = |V(F(H))|+|V_k(H)| = (2k-1)l+2$. The equality for $n_H$ shows that the construction is well-defined for every $n_H > 2k$ such that $n_H \equiv_{2k-1} 2$, but not well-defined for any $n_H \leq 2k$, as then $l < 1$. We have $m_H = k|V_k(H)|+|E(F(H))|= k(kl+1)$, which implies $m_H = \frac{k(kn_H-1)}{2k-1}$. Thus, $H$ lies on the threshold of Bound~\eqref{eq:general}. Basic calculus shows $|V_k(H)| = (k+1)n_H-2m_H = (m_H-n_H+k)/(k-1)$; hence, $H$ satisfies both cases of Bound~\eqref{eq:general} with equality.

\begin{figure}[h!tb]
	\centering
	\subfloat[The graph $H_T(3,4)$, where $T$ is the star graph on $l=4$ vertices. The thick blue subgraph depicts $F$.]{
		\includegraphics[scale=0.8]{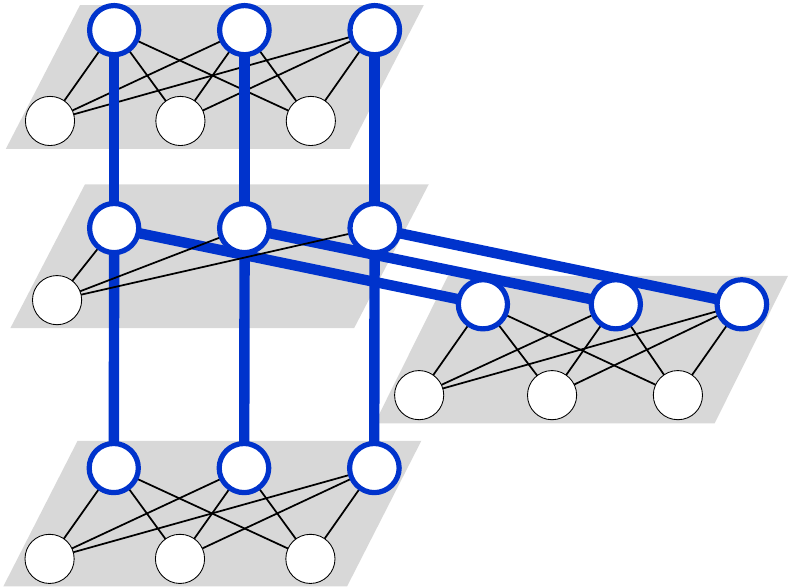}
		\label{fig:ConstructionMader}
	}
	\hspace{0.4cm}
	\subfloat[The graph $H'_T(5,3,3)$, where $T$ is the path on $l=3$ vertices with end vertices $s$ and $t$. The dashed red edges depict the $i=3$ edges that were deleted from $K_{k,k}$ in the layers of $t$ and $s$ as part of a $2$- and a $1$-matching. As $j=1$, the middle row consists only of vertices in $V_k$. The thick blue subgraph depicts $F$, which consists of four isolated vertices.]{
	\makebox[6.5cm]{
		\includegraphics[scale=0.81]{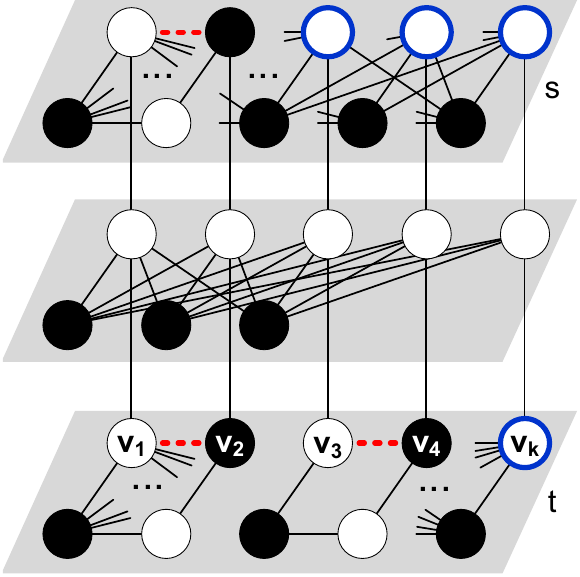}
		\label{fig:SmallM4}
	}}
	\caption{$H_T(k,l)$ and $H'_T(k,l,i)$.}
	\label{fig:Constructions}
\end{figure}

Consider the case $m \leq \frac{k(kn-1)}{2k-1}$ and let $m \equiv_{k(k-1)} k(n-1)-i$ for any $0 \leq i \leq 2\lfloor \frac{k}{2} \rfloor$. We construct a minimally $k$-connected graph satisfying $|V_k|=(k+1)n-2m$. The high-level idea of the construction is to use a modification of $H$ and delete a carefully chosen matching of size $i$ as well as $j$ vertices such that the resulting graph is minimally $k$-connected.

Let $l := \frac{k(n-1)-i-m}{k(k-1)}$. Since $m \equiv_{k(k-1)} k(n-1)-i$, $l$ is an integer. From $m \leq \frac{k(kn-1)}{2k-1}$ follows $l \geq \frac{n-2}{2k-1}-\frac{i}{k(k-1)}$, which implies $l \geq 1-\frac{1}{k-1}$ due to $n > 2k$ and $i \leq k$. Hence, $l \geq 1$. Let $j := l(2k-1)-n+2$; clearly, $j$ is an integer.

We prove that $0 \leq j \leq l$ such that $j = l-1$ implies $i \leq \lfloor \frac{k}{2} \rfloor$, and $j = l$ implies $i=0$. Since $\frac{k(kn-1)}{2k-1} \equiv_{k(k-1)} k(n-1)$, we have $m \equiv_{k(k-1)} \frac{k(kn-1)}{2k-1}-i$. As $m \leq \frac{k(kn-1)}{2k-1}$, this implies $m \leq \frac{k(kn-1)}{2k-1}-i$. Using this bound in the definition of $j$ gives $j \geq 0$. Two elementary calculations on the definitions of $m$, $l$ and $j$ show that $l = \frac{n-2+(j-l)}{2(k-1)}$ and $m = k(n-1)-i-k(k-1)l$. Inserting the former in the latter, we obtain $m \leq \frac{kn}{2}-i-\frac{k}{2}(j-l)$. If $j \geq l$, we conclude $j = l$ and $i=0$, since $m \geq \frac{kn}{2}$ (in addition, $n$ is even in this case, as $n = (2k-1)l-j+2$). If $j = l-1$, $n$ is odd, since $n = (2k-1)l-j+2$. Then $m \leq \frac{kn}{2}+\frac{k}{2}-i$ implies $i \leq \lfloor \frac{k}{2} \rfloor$, as $m \geq \lceil \frac{kn}{2} \rceil$.

Let $T$ be the path on $l$ vertices, let $s$ and $t$ be its end vertices and let $v \in \{s,t\}$. Then the layer of $v$ in $H_T(k,l)$ is $K_{k,k}$; let its two color classes be black and white such that the row vertices $v_1,\ldots,v_k$ are white. In order to describe the construction, we need the following operation of \emph{deleting an $x$-matching}, $0 \leq x \leq \lfloor k/2 \rfloor$, in the layer of $v$ (see Figure~\ref{fig:SmallM4}): For every $1 \leq z \leq x$, replace the (white) row vertex $v_{2z}$ with a black vertex of the $K_{k,k}$ that is not contained in the row of $v$ and delete the edge that joins $v_{2z-1}$ to this black vertex. This way, exactly $x$ edges are deleted from $K_{k,k}$ that form a matching; since both end vertices of every such edge had degree $k+1$ before, this decreases $|V_k|$ by exactly two.

Let $i_t := \min\{i,\lfloor k/2 \rfloor\}$ and $i_s := \max\{0,i-\lfloor k/2 \rfloor\}$; thus, $i_s+i_t=i$. Obtain the graph $H' := H'_T(k,l,i)$ from $H_T(k,l)$ by deleting an $i_t$-matching in the layer of $t$, an $i_s$-matching in the layer of $s$, and one vertex of degree $k$ from each of $j$ layers that are chosen according to the following preference list on their vertices in $T$: inner vertices of $T$, $s$, $t$. This construction is well-defined, as $l \geq 1$ (which is needed for the construction of $H_T(k,l)$) and $0 \leq j \leq l$ such that $j = l-1$ implies $i \leq \lfloor \frac{k}{2} \rfloor$, and $j = l$ implies $i=0$.

By applying Menger's theorem and replacing deleted matching edges with paths (either paths of length three in the same layer or paths that contain exactly two vertices of the layer of the other end vertex of $T$), one can see that $H'$ is $k$-connected.
Since every edge is incident to a vertex of degree $k$ or contained in an edge cut that consists of $k$ edges, $H'$ is minimally $k$-connected. Counting edges and vertices of $H'$ in the same way as done for $H$, we obtain $|V_k(H')| = (k-1)(l+j)+2+2i$, $n_{H'} = (2k-1)l+2-j$ and $m_{H'} = k(kl+1-j)-i$. Thus, expanding $j$ in the equality for $n_{H'}$ shows $n_{H'} = n$, and expanding $j$ and substituting $l$ with $\frac{n-2+j}{2k-1}$ in the equality for $m_{H'}$ shows $m_{H'} = m$. Then $H'$ satisfies $|V_k(H')| = (k+1)n-2m$, as claimed.

Consider the case $m \geq \frac{k(kn-1)}{2k-1}$ and let $m \equiv_{k-1} k(n-1)$. We construct a minimally $k$-connected graph satisfying $|V_k|=(m-n+k)/(k-1)$. In particular, this shows that Bound~\eqref{eq:general} is tight without the ceiling. The high-level idea of the construction is to contract $i$ suitably chosen edges in $H$ such that the resulting graph is minimally $k$-connected, followed by adding sufficiently many new vertices of degree $k$ in order to compensate for the vertex loss.

Since $m \equiv_{k-1} k(n-1)$, $\frac{k(n-1)-m}{k-1}$ is an integer. Let $i \in \{0,\ldots,k-1\}$ such that $\frac{k(n-1)-m}{k-1}+i$ is divisible by $k$; thus, we have $m \equiv_{k(k-1)} k(n-1)+(k-1)i$. Therefore, $l := \frac{k(n-1)-m+(k-1)i}{k(k-1)}$ is an integer.

We prove that $l \geq 1$ and, if $i \neq 0$, $l \geq 2$. Since $G$ is minimally $k$-connected, $m \leq kn-\binom{k+1}{2}$, where equality is only attained for $G = K_{k+1}$, as mentioned before. Since $G = K_{k+1}$ contradicts $n > 2k$, we have $m < kn-\binom{k+1}{2}$. From $m < kn-\binom{k+1}{2} \leq k(n-1)$ and $i \geq 0$ follows $l > 0$ and thus $l \geq 1$. For $i \geq \frac{k}{2}$, we have $m < kn-\binom{k+1}{2} \leq k(n-1)+(k-1)i-k(k-1)$, which implies $l \geq 2$. Consider the remaining case $1 \leq i < \frac{k}{2}$. Since $n \geq 3k-2$, we can use a result of Mader (see e.g.~\cite[Thm.~4.9]{Bollobas2004}), which proves $m \leq kn-k^2$.
Because $i \geq 1$, we have $m \leq kn-k^2 < kn-k^2+(k-1)i = k(n-1)+(k-1)i-k(k-1)$, which shows $l \geq 2$. We conclude for all cases $l \geq 1$ and, if $i \neq 0$, $l \geq 2$.

Let $j := n-2+i-(2k-1)l$; this will be the number of vertices that is added to the contracted graph. Clearly, $j$ is an integer and, since $i < (2k-1)l$, we have $j \leq n-3$. We prove that $j \geq i$.
If $m = \frac{k(kn-1)}{2k-1}$, $m \equiv_{k(k-1)} k(n-1)+(k-1)i$ implies $(2k-1)(k-1)i \equiv_{k(k-1)} 0$ and, as $2k-1$ and $k$ are co-prime, $i = 0$.
Since $m \geq \frac{k(kn-1)}{2k-1}$, $m \geq \frac{k(kn-1)}{2k-1}+(k-1)i$ follows from $m \equiv_{k(k-1)} k(n-1)+(k-1)i$. Inserting this lower bound into the definition of $l$ and using the result in the definition of $j$ gives $j \geq i$. Hence, $0 \leq i \leq j \leq n-3$.

Obtain the graph $H'' := H''_T(k,l,i,j)$ from $H_T(k,l)$ by adding $j$ new vertices of degree $k$ such that the neighbors of every new vertex are in the same row and then contracting $i$ edges of $F$ that are incident to the $k$ copies of a leaf of $T$. This construction is well-defined, as we have $l \geq 1$ and, if there is at least one contraction, the desired $i$ edges in $F$ exist due to $l \geq 2$. By applying Menger's theorem, one can see that $H''$ is $k$-connected.
In addition, $H''$ is minimally $k$-connected, as every edge $e$ is incident to a vertex of degree $k$, contained in an edge cut that consists of $k$ edges, or such that $G-e$ contains a $(k-1)$-separator that consists of $k-1$ copies of the leaf chosen in $T$. Counting edges and vertices as before, we obtain $|V_k(H'')| = (k-1)l+2+j$, $n_{H''} = (2k-1)l+2+j-i$ and $m_{H''} = k(kl+1+j)-i$. Thus, expanding $j$ in the equality for $n_{H''}$ shows $n_{H''} = n$, and expanding $j$ and then $l$ in the equality for $m_{H''}$ shows $m_{H''} = m$. Then $H''$ satisfies $|V_k(H'')| = (m-n+k)/(k-1)$, as claimed. 
\end{proof}

In the tightness proof above, the precondition $n \geq 3k-2$ is used only in the case $m \geq \frac{k(kn-1)}{2k-1}$ for the parity values $1 \leq i < \frac{k}{2}$. Hence, for the remaining values $i=0$ and $\lceil \frac{k}{2} \rceil \leq i \leq k-1$ that satisfy $m \equiv_{k(k-1)} k(n-1)-i$, the weaker precondition $n > 2k$ suffices.

\begin{corollary}\label{cor:boundsmalln}
Bound~\eqref{eq:general} is best possible (even without the ceiling) for every $k \geq 2$, $n > 2k$ and $m \equiv_{k(k-1)} k(n-1)-i$ such that $\lceil \frac{k}{2} \rceil \leq i \leq 2\lfloor \frac{k}{2} \rfloor$.
\end{corollary}

Bound~\eqref{eq:general} implies the best known special-purpose bounds for $k=2$ and $k=3$ (see~\cite[Prop.~2.14+20]{Oxley1981} and~\cite[Fact~81]{Gross2013}) and improves them for every $m < \lfloor \frac{k(kn-1)}{2k-1} \rfloor$. By comparing Bound~\eqref{eq:general} with Mader's Bound~\eqref{eq:Mader}, we obtain immediately that the two bounds match if and only if $m = \frac{k(kn-1)}{2k-1}$. Hence, for the given parities, Mader's bound is only best possible if $m = \frac{k(kn-1)}{2k-1}$; see Figure~\ref{fig:Plot} for a comparison of these two bounds.

\begin{figure}[h!tb]
	\centering
	\subfloat[A 3D-plot for $k=3$. A blue ($m \leq k(kn-1)/(2k-1)$) or green ($m \geq k(kn-1)/(2k-1)$) dot at point $(n,m,|V_k|)$ shows the existence of a graph for which Bound~\eqref{eq:general} is tight. Red dots depict values for which Bound~\eqref{eq:Mader} is tight (neglecting $m$).]{
		\includegraphics[scale=0.29]{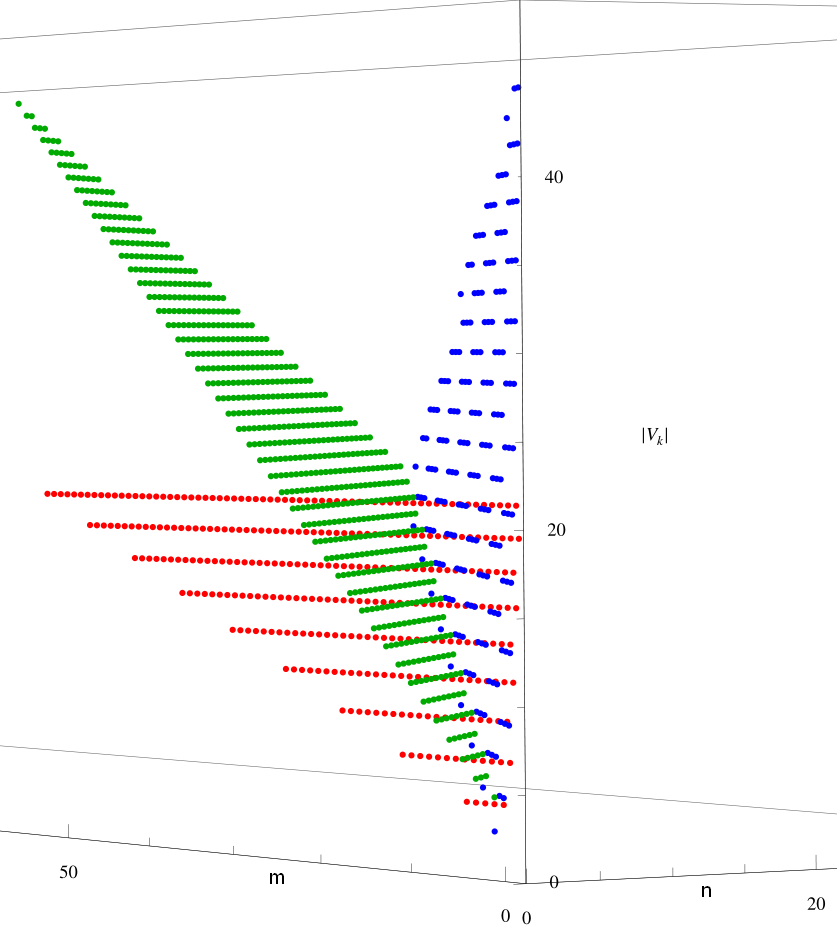}
		\label{fig:Plot1}
	}
	\hspace{0.4cm}
	\subfloat[A 2D-plot for $k=4$ and $n=100$ that shows tight values of Bound~\eqref{eq:general} (green and blue) and Bound~\eqref{eq:Mader} (red) for the relevant ranges of $m$.]{
		\includegraphics[scale=0.18]{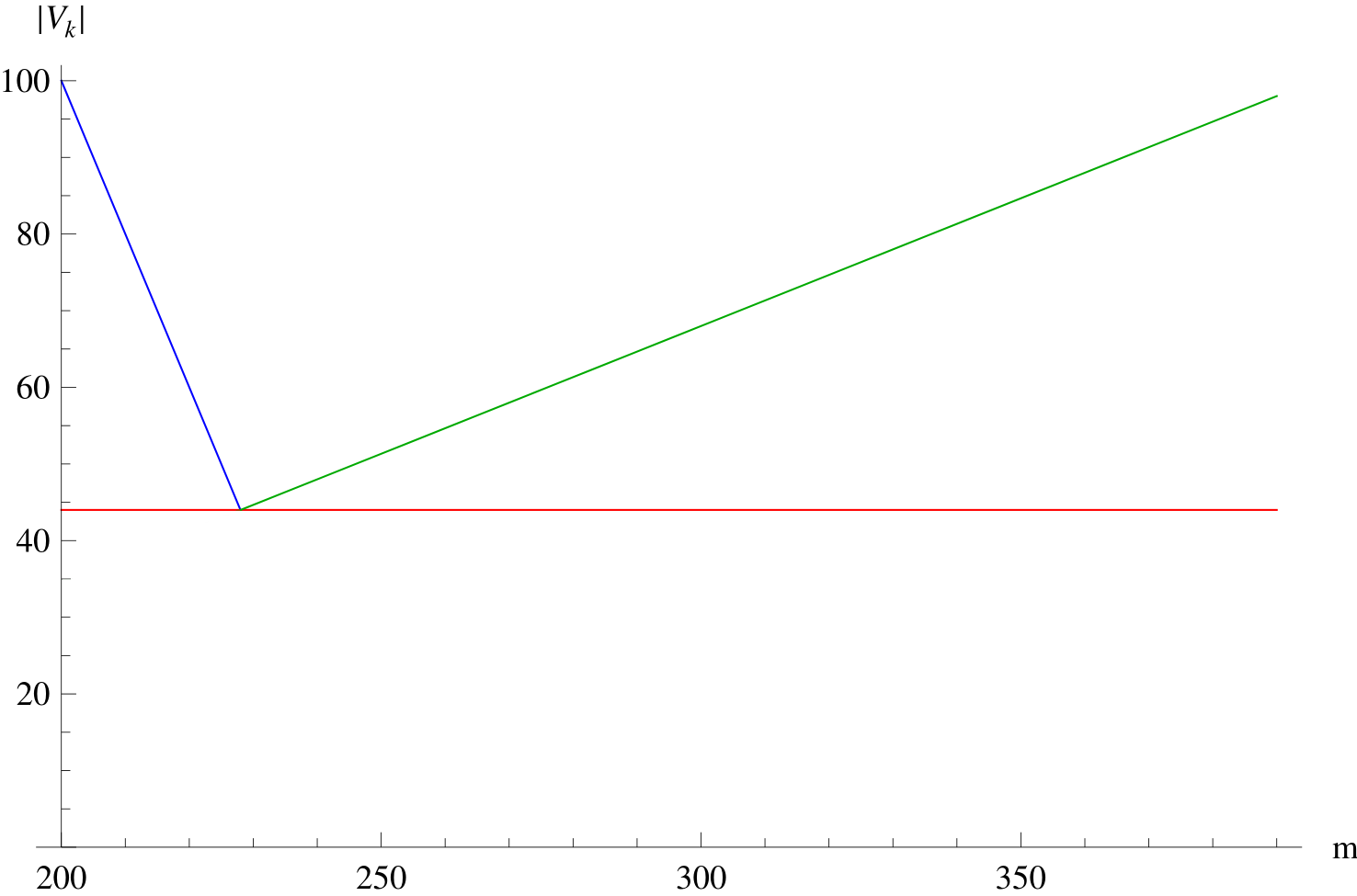}
		\label{fig:Plot2}
	}
	\caption{Comparing tight values of Bounds~\eqref{eq:Mader} and~\eqref{eq:general}.}
	\label{fig:Plot}
\end{figure}

While Corollary~\ref{cor:boundsmalln} shows that Bound~\eqref{eq:general} is tight for $n > 2k$, we leave the problem of determining tight bounds for $n \leq 2k$ as open question. Note that Bound~\eqref{eq:general} is not tight for $n=2k$ and $m=\frac{k(kn-1)}{2k-1}$, as every minimally $k$-connected graph satisfying these constraints has strictly more than $\lceil (m-n+k)/(k-1) \rceil = \lceil k+\frac{k}{2k-1} \rceil = k+1$ vertices in $V_k$ due to~\cite[Satz~4]{Mader1979}.

\paragraph{Acknowledgments.} I wish to thank Thomas Böhme for helpful discussions.




\end{document}